\documentclass[12pt]{amsart}
\usepackage[all]{xy}
\usepackage{amsfonts}
\usepackage{amssymb}

\usepackage[T2A]{fontenc}
\usepackage[cp1251]{inputenc}

\usepackage{amsmath,amssymb,amsthm, latexsym, amscd}

\usepackage{hyperref}

\newtheorem{thm}{Theorem}
\newtheorem{lemma}[thm]{Lemma}
\newtheorem{prop}[thm]{Proposition}

\newtheorem{cor}[thm]{Corollary}
\newtheorem{pr}[thm]{Problem}

\theoremstyle{definition}

\newtheorem{rk}[thm]{Remark}

\newcommand{\rom}[1]{{\em #1}}
\renewcommand{\(}{\rom(}
\renewcommand{\)}{\rom)}
\renewcommand{\:}{\colon}
\newcommand{\sm}{\setminus}
\renewcommand{\ss}{\subset}

\newcommand{\codis}{\operatorname{codis}}
\newcommand{\dis}{\operatorname{dis}}
\newcommand{\Id}{\operatorname{Id}}

\newcommand{\dl}{\delta}

\newcommand{\Lra}{$\Longrightarrow$}
\newcommand{\Llra}{$\Longleftrightarrow$}

\renewcommand{\b}{\beta}
\newcommand{\e}{\varepsilon}
\newcommand{\g}{\gamma}
\renewcommand{\l}{\lambda}

\newcommand{\cE}{\mathcal{E}}

\newcommand{\N}{\mathbb{N}}
\newcommand{\R}{\mathbb{R}}

\makeatletter
\def\namedlabel#1#2{\begingroup\def\@currentlabel{#2}\label{#1}\endgroup}
\makeatother

\makeatletter
\def\.{.\spacefactor\@m}
\makeatother

\begin{document}

\title{Mazur--Ulam Theorem With Gromov-Hausdorff Distance.}

\author{ S.\,A.~Bogaty , A.\,A.~Tuzhilin }

\date{}
\maketitle

\begin{abstract}
It is shown that two Banach spaces are linearly isometric if and only if the Gromov--Hausdorff distance between them is finite, in particular, zero. The proof is compilative and relies on results obtained by many researchers on the approximability of almost-surjective almost-isometries by linear surjective isometries. In the finite-dimensional case, previously obtained by I.~Mikhailov, a simpler proof under weaker assumptions is given. In the finite-dimensional case, a criterion for isometry in terms of finite (compact) subsets is also given.


\thanks{A.A. Tuzhilin's work was supported by a grant from the Russian Science Foundation (project 25-21-00152) and the Sino-Russian Mathematical Center at Peking University.}

\emph{Semeon A. Bogatyi}, \email{bogatyi@inbox.ru}

\emph{Alexey A. Tuzhilin}, \email{tuz@mech.math.msu.su}
\end{abstract}

The famous Mazur--Ulam theorem of 1932 states that \emph{every surjective isometric mapping $f\:V\to W$ of real normed spaces is affine}~\cite{MU32}. The simplest proof of the theorem was given in 2003 by J.~V\"ais\"al\"a~\cite{V03} and in 2012 by B.~Nika~\cite{N12}. In the case of a strictly normed space $W$, the statement holds even without the  surjectivity assumption on the mapping~\cite{B71}.

In 1932, in his famous monograph ``Theory of Linear Operations'' Stefan Banach connec\-ted the homeomorphism of metrizable compact spaces with the (linear) isometry of spaces of continuous functions on them~\cite[Chapter XI, \S 4, Theorem 3]{B32}. In 1937, Marshall Harvey Stone~\cite[Theorem 83]{S37} and in 1942, Samuel Eilenberg~\cite[Theorem 7.2]{E42} removed the assumption of metrizability of spaces.

\begin{thm}\label{thm:compactspaces}
For compact Hausdorff spaces $X$ and $Y$, the following conditions are equiva\-lent\/\rom:
\begin{enumerate}
\item the spaces $X$ and $Y$ are homeomorphic\/\rom;
\item the Banach spaces $V=C(X)$ and $W=C(Y)$ of real continuous functions with $\sup$-norm are linearly isometric\/\rom;
\item the metric spaces $V=C(X)$ and $W=C(Y)$ are isometric.
\end{enumerate}
\end{thm}

One of the ways to strengthen the Mazur--Ulam theorem arose from the 1945 theorem of Hyers--Ulam~\cite{HU45}. The distance between points $x$ and $x'$ of an arbitrary metric space will be denoted by $|xx'|$. For a mapping $f\:X\to Y$ of metric spaces, the \emph{distortion\/} is
\begin{equation}\label{eq:1}
\dis f=\sup_{x,x'\in X}\Bigl|\bigl|f(x)f(x')\bigr|-|xx'|\Bigr|.
\end{equation}
A mapping $f$ is called an \emph{$\e$-isometry\/} if and only if $\dis f\le\e$.

The Hyers--Ulam theorem states that \emph{for every surjective $\e$-isometry
$$
f\:V\to W,\ V=W
$$
of a complete Euclidean vector space onto itself, for which $f(0)=0$, there exists a bijective linear isometry $U\:V\to W$ such that
\begin{equation}\label{eq:2}
\bigl\|f(v)-U(v)\bigr\|\le M=10\e\ \text{ for all $v\in V$}.
\end{equation}
This isometry $U$ is given by the formula $U(v)=\lim_{n\to\infty}\frac1{2n}f(2nv)$.}

D.G.~Bourgin (1946) proved that such a linear isometry of $U$ exists whenever $W$ belongs to the class of uniformly convex Banach spaces, including $L_p[0,1]$ for $1<p<\infty$. Here we can take $M=12$ in (\ref{eq:2}). In a subsequent paper by Hyers and Ulam (1947), a positive result was obtained for the case where $V=C(X)$ and $W=C(Y)$ are Banach spaces of real continuous functions on compact Hausdorff spaces with the sup-norm, provided that $f$ is a homeomorphism. In particular, we can take $M=21$. This study was continued by D.G.~Bourgin~\cite{B46} (1949), who showed, in particular, that the assumptions of continuity and one-to-oneness of $f$ can be lifted. This result is a significant generalization of the classical Banach--Stone theorem.

In 1975, R.D.~Bourgin considered the Hyers-Ulam problem and studied mainly the finite-dimensional case. He showed that in this case the answer is positive if we assume that the set of extremal boundary points of the unit ball in $V$ or $W$ is completely disconnected. D.G.~Bourgin (1978) proved that the Hyers--Ulam theorem is true for $\dim V=\dim W=2$.

P.M.~Gruber~\cite{G78} (1978) showed that if for a surjective $\e$-isometry $f\:V\to W$ satisfying $f(0)=0$ there exists an approximating linear isometry $U$, $U(0)=0$, with some constant $M$ on the right-hand side of~(\ref{eq:2}), then this isometry is unique and for it we can assume that $M=5\e$. If it is additionally known that $U$ is continuous, then $M=3\e$. He also proved the existence of an approximating linear isometry for the $\e$-isometry of finite-dimensional spaces $V$ and $W$.

In 1983, Gevirtz generalized the Hyers--Ulam theorem to arbitrary Banach spaces $V$ and $W$. Finally, in 1995, M.~Omladi\v{c} and P.~\v{S}emrl showed that $M=2\e$ is a sharp constant in~(\ref{eq:2}) for general Banach spaces. In 1998, \v{S}emrl proved that $M=\g_V\e$ is valid, where Gevirtz is Young's constant. Other proofs of Gevirtz's theorem are published in~\cite{R01} and~\cite{G11.2}.

The assumption of surjectivity is a necessary condition for these results. Namely, D.H.~Hyers and S.M.~Ulam~\cite{HU45} already gave an example of an $\e$-isometry $f$ from the real line to the Euclidean plane such that $\Bigl\{\bigl|f(x)-U(x)\bigr|:x\in\R\Bigr\}$ is an unbounded set for any isometry $U\:\R\to\R^2$.

In fact, already R.D.~Bourgin noted~\cite[p. 312]{B75}: \emph{Although the following results are given for surjective $\e$-isometries, it should be noted that only minor changes are required to adjust the proofs if ``surjective'' is replaced by ``$d_H(f(V),W)<\infty$''.}

But, apparently, the first explicitly formulated result for a non-surjective mapping was published in 1997 and concerns the case of finite-dimensional spaces of one dimension $V=W=\ell_2^n$, see~\cite{BS97}. We present a formulation of a somewhat final result of S.J.~Dilworth~\cite[Proposition 2]{D99} from 1999, which we use here.

\begin{thm}\label{thm:D}
If $\e,\,\dl\ge0$ and $f\:V\to W$ is an $\e$-isometry between Banach spaces satisfying the condition $f(0)=0$ and mapping $\dl$-surjectively the space $V$ onto some closed subspace $L\ss W$ \($d_H\bigl(L,f(V)\bigr)\le\dl$\), then there exists a bijective linear isometry $U\:V\to L$ such that
\begin{equation}\label{eq:3}
\bigl\|f(v)-U(v)\bigr\|\le M=12\e+5\dl\ \text{for all $v\in V$}.
\end{equation}
\end{thm}

In 2000, J.~Tabor independently proved the stated theorem for $L=W$ with the constant $M=2\e+35\dl$ on the right-hand side (\ref{eq:3}). Recently, a simpler reduction of the Dilworth--Tabor theorem to the Gevirtz theorem was obtained. In 2003, P.~\v{S}emrl and J.~V\"ais\"al\"a showed that $M=2\e+2\dl$ can be taken on the right-hand side of~(\ref{eq:3}). They also showed that the universal estimate with $M=2\e$ holds for $L=W$.

For mappings $f\:X\to Y$ and $g\:Y\to X$ of metric spaces, \emph{codistortion\/} is
\begin{equation}\label{eq:4}
\codis(f,g)=\sup_{x\in X,y\in Y}\Bigl|\bigl|f(x)y\bigr|-\bigl|x\,g(y)\bigr|\Bigr|.
\end{equation}
Definitions~(\ref{eq:1}) and~(\ref{eq:4}) do not assume continuity of the possibly multivalued mappings under consideration. In the case of a multivalued mapping, $f(x)$ and $g(y)$ can take any points from these sets.

For any two metric spaces $X$ and $Y$, the \emph{Gromov--Hausdorff distance\/} is the quantity
$$
d_{GH}(X,Y)=\inf_{f,g}d_H\bigl(f(X),g(Y)\bigr),
$$
where $f\:X\to Z$ and $g\:Y\to Z$ are isometric embeddings into all possible metric spaces $Z$. There is a second important approach to defining the Gromov--Hausdorff distance~\cite{BBI}:
\begin{equation}\label{eq:5}
d_{GH}(X,Y)=\frac12\inf_{f,g}\Bigl\{\max\bigl\{\dis f,\,\codis(f,g),\,\dis g\bigr\}\,\big\vert\,f\:X\to Y,\,g\:Y\to X\Bigr\}.
\end{equation}

For any two metric spaces $X$ and $Y$, we define the \emph{continuous Gromov--Hausdorff distance\/} as in~\cite{LMS23, BT26}:
\begin{equation*}
d_{GH}^c(X,Y)=\frac12\inf_{\substack{\text{$f$, $g$ ---}\\ \text{continuous}}}\Bigl\{\max\bigl\{\dis f,\,\codis(f,g),\,\dis g\bigr\}\,\big\vert\,f\:X\to Y,\,g\:Y\to X\Bigr\}.
\end{equation*}
Obviously, $d_{GH}(X,Y)\le d_{GH}^c(X,Y)$.

\begin{thm}\label{thm:equivalent}
For real Banach spaces $V$ and $W$, the following conditions are equivalent\/\rom:
\begin{enumerate}
\item\namedlabel{thm:equivalent:1}{(1)} the spaces $V$ and $W$ are linearly isometric\/\rom;
\item\namedlabel{thm:equivalent:2}{(2)} the spaces $V$ and $W$ are isometric\/\rom;
\item\namedlabel{thm:equivalent:3}{(3)} $d_{GH}^c(V,W)=0$\rom;
\item[$(3')$]\namedlabel{thm:equivalent:3'}{(3')} $d_{GH}^c(V,W)<\infty$\rom;
\item[$(4')$]\namedlabel{thm:equivalent:4'}{(4')} $d_{GH}(V,W)<\infty$\rom;
\item[$(5)$]\namedlabel{thm:equivalent:5}{(5)} the unit balls $B_V$ and $B_W$ are isometric.
\end{enumerate}
\end{thm}

\begin{proof}
Implications \ref{thm:equivalent:1}\Lra\ref{thm:equivalent:2}\Lra\ref{thm:equivalent:3}\Lra$\ref{thm:equivalent:3'}$\Lra$\ref{thm:equivalent:4'}$ and \ref{thm:equivalent:1}\Lra\ref{thm:equivalent:5} are obvious.

$\ref{thm:equivalent:4'}$\Lra\ref{thm:equivalent:1}. From the condition $d_{GH}(V,W)<D<\infty$ it follows that there exist mappings $f\:V\to W$ and $g\:W\to V$ such that $\dis f<D$ and $\codis(f,g)<D$. The first inequality precisely means that the mapping $f$ is a $D$-isometry. The second inequality implies that for every vector $w\in W$, the inequality $\Bigl|f\bigl(g(w)\bigr)\,w\Bigr|<D$ holds, i.e., the mapping $f$ is $D$-surjective. By the Dilworth--Tabor theorem there is a linear surjective isometry $U\:V\to W$.

The implication \ref{thm:equivalent:5}\Lra\ref{thm:equivalent:1} is precisely Mankiewicz's theorem~\cite{M72}.
\end{proof}

\begin{cor}\label{cor:equivalent}
For a real Banach space $V$ and a normed space $W$, the following conditions are equivalent\/\rom:
\begin{enumerate}
\item the space $W$ is linearly isometrically embedded in $V$ as a dense subspace\/\rom;
\item\refstepcounter{enumi} the space $W$ is isometrically embedded in $V$ as a dense subset\/\rom;
\item $d_{GH}(V,W)=0$\rom;
\item[($4'$)] $d_{GH}(V,W)<\infty$.
\end{enumerate}
\end{cor}

\begin{rk}\label{rk:equivalent}
If at least one one the spaces $V$ or $W$ is finite-dimensional, an essential part of Theorem~\ref{thm:equivalent} and Corollary~\ref{cor:equivalent} (implication $\ref{thm:equivalent:4'}$\Lra\ref{thm:equivalent:1}) was proven by I.~Mikhailov. His proof is quite complicated, but it does not rely on the even more complicated theory of $\e$-isometry.

We present a third proof to emphasize that even in the finite-dimensional case, the implication $\ref{thm:equivalent:4'}$\Lra\ref{thm:equivalent:1}, i.e., the theorem on the existence of an isometry, is significantly simpler than the isometry approximation theorem. In this regard, it is appropriate to recall the general Ulam problem of approximating of a rather good object with a perfect object.
\end{rk}

\begin{prop}\label{prop:equivalent}
For conical subsets of real normed spaces $K_V\ss V$ and $K_W\ss W$, the following conditions are equivalent\/\rom:
\begin{enumerate}
\setcounter{enumi}{3}
\item\label{prop:equivalent:4} $d_{GH}(K_V,K_W)=0$\rom;
\item[$(4')$]\namedlabel{prop:equivalent:4'}{(4')} $d_{GH}(K_V,K_W)<\infty$.
\end{enumerate}
\end{prop}

\begin{proof}
$\ref{prop:equivalent:4'}$\Lra\ref{prop:equivalent:4}. Let $d_{GH}(K_V,K_W)<D<\infty$. Since for conical subsets we have $K=\l K$ for every $\l>0$, we get $d_{GH}(K_V,K_W)=d_{GH}(\l K_V,\l K_W)=\l d_{GH}(K_V,K_W)$. Therefore, $d_{GH}(K_V,K_W)=0$.
\end{proof}

For any two normed spaces $V$ and $W$, the Kadets distance is defined in~\cite{K75} as
$$
d_K(V,W)=\inf_{f,g}d_H\bigl(f(B_V),g(B_W)\bigr),
$$
where $f\:V\to Z$ and $g\:W\to Z$ are linear isometric embeddings into all possible normed spaces $Z$. Obviously, $d_{GH}(B_V,B_W)\le d_K(V,W)$.

For any two isomorphic normed spaces $V$ and $W$, the Banach–-Mazur distance is defined as follows~\cite[note to Chapter XI, \S 6]{B32}:
$$
d_{BM}(V,W)=\log\bigl(\inf_f\|f\|\cdot\|f^{-1}\|\bigr),
$$
where $f$ runs over all possible isomorphisms between $V$ and $W$. For non-isomorphic spaces, we set $d_{BM}(V,W)=\infty$.

\begin{thm}\label{thm:equivalentfinite}
For a normed space $V$ and a finite-dimensional normed space $W$, the following conditions are equivalent\/\rom:
\begin{enumerate}
\item\label{thm:equivalentfinite:1} the spaces $V$ and $W$ are linearly isometric\/\rom;
\refstepcounter{enumi}
\item\label{thm:equivalentfinite:3} $d_{GH}^c(V,W)=0$\rom;
\item[$(3')$]\namedlabel{thm:equivalentfinite:3'}{(3')} $d_{GH}^c(V,W)<\infty$\rom;
\item\label{thm:equivalentfinite:4} $d_{GH}(V,W)=0$\rom;
\item[$(4')$]\namedlabel{thm:equivalentfinite:4'}{(4')} $d_{GH}(V,W)<\infty$\rom;
\item\label{thm:equivalentfinite:5} single the balls $B_V$ and $B_W$ are isometric \/\rom;
\item\label{thm:equivalentfinite:6} $d_{BM}(V,W)=0$\/\rom;
\item\label{thm:equivalentfinite:7} $d_{GH}^c(B_V,B_W)=0$\rom;
\item\label{thm:equivalentfinite:8} $d_K(V,W)=0$\rom;
\item\label{thm:equivalentfinite:9} $d_{GH}(B_V,B_W)=0$.
\end{enumerate}
\end{thm}

\begin{proof}
Implications
\begin{center}
(\ref{thm:equivalentfinite:1})\Lra(\ref{thm:equivalentfinite:3})\Lra(\ref{thm:equivalentfinite:4})\Lra(\ref{thm:equivalentfinite:9}), (\ref{thm:equivalentfinite:1})\Lra$\ref{thm:equivalentfinite:3'}$\Lra$\ref{thm:equivalentfinite:4'}$,
(\ref{thm:equivalentfinite:1})\Lra(\ref{thm:equivalentfinite:5})\Lra (\ref{thm:equivalentfinite:7})\Lra(\ref{thm:equivalentfinite:9}),\\ (\ref{thm:equivalentfinite:1})\Lra(\ref{thm:equivalentfinite:6}) (\ref{thm:equivalentfinite:1})\Lra(\ref{thm:equivalentfinite:8})\Lra(\ref{thm:equivalentfinite:9})
\end{center}
are obvious.

The implication $\ref{thm:equivalentfinite:4'}$\Lra(\ref{thm:equivalentfinite:4}) follows from Proposition~\ref{prop:equivalent}.

The implication (\ref{thm:equivalentfinite:6})\Lra(\ref{thm:equivalentfinite:1}) is mentioned in~\cite[note To Chapter XI, \S 6]{B32} by S.~Banach as the result obtained together with S.~Mazur.

(\ref{thm:equivalentfinite:9})\Lra(\ref{thm:equivalentfinite:5}). For a finite-dimensional space $W$, the unit ball $B_W$ is compact and, in particular, totally bounded. Then the unit ball $B_V$ is also totally bounded. Consequently, by the Riesz theorem, the space $V$ is also finite-dimensional. Hence the unit ball $B_V$ is compact. For compact spaces, the equality of the distance between them to zero implies that they are isometric. Therefore, the unit balls $B_V$ and $B_W$ are isometric.
\end{proof}

\begin{rk}\label{rk:infinite}
Let us note that finite dimensionality space $W$ is used only the proofs of implications (\ref{thm:equivalentfinite:6})\Lra(\ref{thm:equivalentfinite:1}) and (\ref{thm:equivalentfinite:9})\Lra(\ref{thm:equivalentfinite:5}). In general (infinite-dimensional) case, Theorem~\ref{thm:equivalent} cannot contain the conditions~(\ref{thm:equivalentfinite:8}) and~(\ref{thm:equivalentfinite:9}). M.I.~Ostrovskii~\cite[Theorem 1, c)]{O84} constructed separable Banach spaces $V$ and $W$ such that $d_{GH}(B_V,B_W)=d_K(V,W)=0$, however, $d_{BM}(V,W)=\infty$.
\end{rk}

\begin{lemma}[{P.L.~Renz~\cite[Lemma 2.8]{B75}}]\label{lemma:1}
Let $V$ and $W$ be normed linear spaces, and $f\:V\to W$ be an arbitrary $\e$-isometry. Then there exists a continuous $4\e$-isometry $f^{*}\:V\to W$ such that $\bigl\|f^{*}(v)-f(v)\bigr\|\le2\e$ for every $v\in V$.
\end{lemma}

\begin{lemma}\label{lemma:2}
Let $V$ and $W$ be normed linear spaces such that $\min\bigl\{\dim V,\dim W\bigr\}<\infty$, and let $f\:V\to W$ be a mapping such that $\dis f<\infty$. Then $\dim V\le\dim W$.
\end{lemma}

\begin{proof}
By Lemma~\ref{lemma:1} we can assume that $f$ is continuous. Suppose that $\infty>\dim V>\dim W$. Then by the Borsuk--Ulam theorem, for every $R>0$ there exists $v\in V$ such that $\|v\|=R$ and $f(v)=f(-v)$. It follows that $\dis f\ge\Bigl|\bigl\|f(v)-f(-v)\bigr\|-\bigl\|v-(-v)\bigr\|\Bigr|=2R$. A contradiction.

Let $\infty=\dim V>\dim W$. Take a subspace $L\ss V$ such that $\dim L=\dim W+1$. Then, according to the case already discussed, $\dim W\ge\dim L>\dim W$.
\end{proof}

\begin{lemma}[{DallasWebster~\cite[Proposition 4.1]{B75}}]\label{lemma:3}
Let $V$ and $W$ be normed linear spaces, $f\:V\to W$ be a continuous mapping, and $\dim W<\infty$, $\dim W\le\dim V$, $\dis f<\infty$. Then $\dim W=\dim V$ and $f$ is a surjective mapping.
\end{lemma}

\begin{proof}
According to Lemma~\ref{lemma:2} we have $\dim V=\dim W$ and then we apply~\cite[Pro\-po\-si\-tion 4.1]{B75}.
\end{proof}

\begin{thm}\label{thm:equivalentfinite2}
For a normed space $V$ and a finite-dimensional normed space $W$, the following conditions are equivalent\/\rom:
\begin{enumerate}
\item\label{thm:equivalentfinite2:1} the spaces $V$ and $W$ are linearly isometric\/\rom;
\item[$(3'')$]\namedlabel{thm:equivalentfinite2:3''}{(3'')} $\dim W\le\dim V$ and there exists a mapping $f\:V\to W$ such that $\dis f<\infty$.
\end{enumerate}
\end{thm}

\begin{proof}
The implication (\ref{thm:equivalentfinite2:1})\Lra$\ref{thm:equivalentfinite2:3''}$ is obvious.

$\ref{thm:equivalentfinite2:3''}$\Lra(\ref{thm:equivalentfinite2:1}). By Lemma~\ref{lemma:1} we can assume that the mapping $f$ is continuous. By Lemma~\ref{lemma:3} $\dim W=\dim V$ and the mapping $f$ is surjective. Further, we apply~\cite[Theorem 3]{G78}.
\end{proof}

\begin{rk}\label{rk:isometry}
Note that all the results presented here implicitly assume that the problem is to find (construct) a surjective linear isometry. However, four problems are natural:
\begin{enumerate}
\item finding a surjective approximating isometry;
\item finding a surjective isometry;
\item finding (linear) approximating isometry;
\item finding a (linear) isometric embedding.
\end{enumerate}
There are separate tasks:
\begin{enumerate}
\setcounter{enumi}{4}
\item finding conditions that imply the linearity of isometry;
\item\label{rk:isometry:6} finding conditions for the existence of the Figiel operator $T$~\cite{F68} and the existence of its inverse linear operator (isometry).
\end{enumerate}
Here it is appropriate to recall the problem of A.D. Alexandrov~\cite{A70, RS93}.
\begin{enumerate}
\setcounter{enumi}{6}
\item When is every mapping $f\:X\to Y$ (not necessarily continuous) that preserves distance $1$ an isometry?
\end{enumerate}

In the paper~\cite{CDZ13} four other problems are formulated, distinguished by the properties of a given $\e$-isometry:
\begin{enumerate}
\item $f$ is surjective and $\e=0$;
\item $f$ is not surjective and $\e=0$;
\item $f$ is surjective and $\e\ne0$;
\item $f$ is not surjective and $\e\ne0$.
\end{enumerate}
\end{rk}

In 1953, F.S.~Beckman and D.A.~Quarles, Jr.~obtained the following result, which has already become a textbook example. \emph{A transformation of the Euclidean $n$-space $E^n$ \($2\le n<\infty$\) that preserves a unique nonzero length must be a Euclidean\/ \(rigid\/\) motion of $E^n$.} Isometric property follows from the preservation of the unit area of a triangle~\cite[Theorem 5]{BBF01}. An excellent review of isometricity criteria was given by J.A.~Lester~\cite{L95}. As a solution to a weakened version of Alexandrov's problem, we cite a beautiful result of A.~Fogt~\cite{F73}. \emph{Let $V$ and $W$ be two normed spaces, $f\:V\to W$, $f(0)=0$, be a continuous surjective mapping preserving distance equalities\/ \(i.e., $\|xy\|=\|zw\|$ implies $\bigl\|f(x)-f(y)\bigr\|=\bigl\|f(z)-f(w)\bigr\|$\). Then $f=\l U$, where $\l$ is a nonzero real number, and $U$ is a linear isometry of $V$ onto $W$.}

We already noted in the first paragraph that every isometry $f\:V\to W$ (not necessarily surjective) into a strictly convex normed space $W$ is affine. The fundamental result in the non-surjective case was obtained by T.~Figiel~\cite{F68}. \emph{Let $V$ and $W$ be two Banach spaces, $f\:V\to W$, $f(0)=0$, be an isometry. Then there exists a unique bounded linear operator $T(f)\:\overline{\operatorname{span}}\bigl(f(V)\bigr)\to V$ with $\bigl\|T(f)\bigr\|=1$ such that $T(f)\circ f=\Id_V$}.

In 2003, G.~Godefroy and N.~Kalton~\cite{GK03, G10} explored the relationship between isometry and linear isometry (cf\. Question~(\ref{rk:isometry:6}) of Remark~\ref{rk:isometry}), and proved the following deep theorem: \emph{Let $V$ and $W$ be two Banach spaces, $f\:V\to W$, $f(0)=0$ be an isometry, and $T$ be the Figel operator of $f$.
\begin{enumerate}
\renewcommand{\theenumi}{\Roman{enumi}}
\item If $V$ is a separable Banach space, then there exists a linear isometry $S\:V\to\overline{\operatorname{span}}\bigl(f(V)\bigr)$ such that $T\circ S=\Id_V$ and $W$ contains a set of well-complemented subspaces that are linearly isometric to $V$.
\item If $V$ is a non-separable weakly compactly generated Banach space, then there exists a Banach space $W$ such that there exists
nonlinear isometry $f\:V\to W$, but $V$ is not linearly isomorphic to any subspace of $W$.
\end{enumerate}
}

A similar question was studied by Y.~Dutrieux  and G.~Lancien. We will cite one of the questions they posed~\cite[Question (3) on p. 500]{DL08}.

\begin{pr}\label{pr:embedding}
If two separable Banach spaces have the same compact\/ \(finite\/\) subsets up to isometry, then are they isometrically embedded in each other\/\rom?
\end{pr}

For finite-dimensional spaces, the answer is positive.

\begin{prop}\label{prop:finite}
Let $\bigl(V,\|\cdot \|_V\bigr)$ and $\bigl(W,\|\cdot \|_W\bigr)$ be normed spaces such that $\dim W<\infty$ and each finite subset $S=\{v_0,v_1,\ldots,v_m\}\ss V$ with $ \|v_0-v_1\|_V=\ldots=\|v_0-v_m\|_V=1$ can be isometrically embedded into $W$. Then $\dim V\le\dim W$.
\end{prop}

\begin{proof}
Suppose that $\dim V>\dim W$. Consider a subspace $Z$ of $V$ such that $\dim Z=\dim W+1$. For every positive integer $n$, fix an $n^{-1}$-net $S(n)$ in the unit sphere $S_Z$ of $Z$. Suppose that for every $n\in\N$, the finite subset $Z(n)=S(n)\cup\{0_Z\}\ss Z$ is isometrically embedded into $W$: $i_n\:Z(n)\to W$. Applying the shift by the vector $\overrightarrow{i_n(0_Z)0_W}$, we can assume without loss of generality that $i_n(0_Z)=0_W$. Then $i_n\bigl(S(n)\bigr)$ is a finite subset of the unit sphere $S_W$. The resulting sequence of compact subsets of the sphere $S_W$ forms a sequence of points in the compact space of compact subsets $S_W$ with Hausdorff distance $d_H$. We choose a convergent subsequence $\lim_{i\to\infty}i_{n_i}\bigl(S(n_i)\bigr)=K\ss S_W$. Then
\begin{multline*}
d_{GH}(S_Z,K)\le d_{GH}\bigl(S_Z,S(n_i)\bigr)+d_{GH}\bigl(i_{n_i}\bigl(S(n_i)\bigr),K\bigr)\le\\
\le n_i^{-1}+d_{GH}\bigl(i_{n_i}\bigl(S(n_i)\bigr),K\bigr)\to_{i\to\infty}0.
\end{multline*}
The resulting chain of inequalities means that $d_{GH}(S_Z,K)=0$. Thus, the compact sets $S_Z$ and $K\ss W$ are isometric, which contradicts Borsuk's theorem.
\end{proof}

\begin{thm}\label{thm:equivalentfinitedistance}
For a normed space $V$ and a finite-dimensional normed space $W$, the following conditions are equivalent\/\rom:
\begin{enumerate}
\item\label{thm:equivalentfinitedistance:1} the spaces $V$ and $W$ are linearly isometric\rom;
\item[$(5')$]\namedlabel{thm:equivalentfinitedistance:5'}{(5')} $\dim W\le\dim V$ and there is an isometric embedding $f\:B_V\to W$\rom;
\item[$(10)$]\namedlabel{thm:equivalentfinitedistance:10}{(10)} the spaces $V$ and $W$ have the same compact subsets up to isometry\rom;
\item[$(10')$]\namedlabel{thm:equivalentfinitedistance:10'}{(10')} $\dim W\le\dim V$ and each finite subset of $V$ isometrically is invested in $W$.
\end{enumerate}
\end{thm}

\begin{proof}
The implications (\ref{thm:equivalentfinitedistance:1})\Lra$\ref{thm:equivalentfinitedistance:10}$ and (\ref{thm:equivalentfinitedistance:1})\Lra$\ref{thm:equivalentfinitedistance:5'}$\Lra$\ref{thm:equivalentfinitedistance:10'}$ are obvious.

$\ref{thm:equivalentfinitedistance:10}$\Lra$\ref{thm:equivalentfinitedistance:5'}$. Let $f\:B_W\to V$ be an isometric embedding of the compact set $B_W$. By the Borsuk--Ulam theorem, $\dim W\le\dim V$. Let $\dim W<\dim V$. Consider a subspace $L\ss V$ such that $\dim L=\dim W+1$. Let $f\:B_L\to W$ be an isometric embedding of the compact set $B_L$. By the Borsuk--Ulam theorem, $\dim L\le\dim W=\dim L-1$. The obtained contradiction shows that $\dim W=\dim V$ and, therefore, there exists an isometric embedding $g\:B_V\to W$ of the compact set $B_V$.

$\ref{thm:equivalentfinitedistance:5'}$\Lra(\ref{thm:equivalentfinitedistance:1}). Let $\dim W<\dim V$. Consider a subspace $L\ss V$ such that $\dim L=\dim W+1$. By the Borsuk--Ulam theorem, applied to the isometric embedding $f_{|B_L}\:B_L\to W$, the inequality $\dim L\le\dim W=\dim L-1$ holds. The resulting contradiction shows that $\dim W=\dim V$. Under the isometry $f\:B_V\to W$, the unit sphere $S_V(0,1)$ with center $0\in V$ is mapped to the unit sphere $S_W\bigl(f(0),1\bigr)$ with center $f(0)\in W$: $f\bigl(S_V(0,1)\bigr)\ss S_W\bigl(f(0),1\bigr)$. It follows from the Borsuk--Ulam theorem that $f\bigl(S_V(0,1)\bigr)=S_W\bigl(f(0),1\bigr)$. This means that $f\bigl(B_V(0,1)\bigr)=B_W\bigl(f(0),1\bigr)$ and by Mankiewicz's theorem (by the fulfillment of condition $\ref{thm:equivalent:5}$ of the theorem~\ref{thm:equivalent} and condition $\ref{thm:equivalentfinite:5}$ of the theorem~\ref{thm:equivalentfinite}) the spaces $V$ and $W$ are linearly isometric.

$\ref{thm:equivalentfinitedistance:10'}$\Lra$\ref{thm:equivalentfinitedistance:5'}$. By Proposition~\ref{prop:finite}, $\dim W=\dim V$. Therefore, the unit ball $B_V$ is compact and, by repeating the proof of Proposition~\ref{prop:finite}, we construct an isometric embedding $f\:B_V\to W$.
\end{proof}

In relation with Proposition~\ref{prop:finite}, for $\dim V>\dim W$, natural questions arise of describing minimal finite subsets of $V$ that are not embeddable in $W$ and of estimating the distances of such subsets to subsets of $W$. A metric space $S$ is called \emph{equilateral\/} (\emph{regular simplex\/}) if every pair of points in $S$ is equidistant. K.M.~Petty proved~\cite[Theorem 4]{P71} that the cardinality of a regular simplex in an $n$-dimensional normed space does not exceed $2^n$. Therefore, in the normed space $\ell_\infty^{n+1}$ there is a $(2^n+1)$-vertex simplex that is not isometrically embeddable in any $n$-dimensional normed space.

Note that M.~Gromov~\cite[Theorem 3.3]{G07} showed that compact metric spaces are isometric if and only if their finite subsets are the same up to isometry. F.~Memoli showed~\cite[Theorem 5.1]{M12} that the modified Gromov--Hausdorff distance between compact metric spaces is determined by the distances between their finite subsets. P.J.~Olver established~\cite{O01} that smooth plane curves $X$ and $Y$ are rigidly isometric if and only if they have the same four-point subsets up to isometry. P.J.~Olver also proved that two smooth surfaces $X$ and $Y$ embedded in $\R^3$ are rigidly isometric if and only if they have the same four-point subsets up to isometry. M.~Boutin and G.~Kemper showed~\cite{BK04, BK07} that the distance distribution and the triangle distribution metrically characterize smooth plane curves.

In the papers~\cite{FSV02},~\cite[Theorem 3.3]{ZZL16} very simple and almost necessary conditions for the linearity of a non-surjective mapping are proposed.

We fix some norm $\|\cdot\|$ on a real two-dimensional vector space $\R^2$, and for a normed space $V$ we define the norm on $W_V=V\oplus\R$ as follows: $\bigl\|(v,t)\bigr\|_W=\Bigl\|\bigl(\|v\|_V,t\bigr)\Bigr\|$. In many papers~\cite{HU45},~\cite[Theorem 2]{G78},~\cite{FSV02} it is shown that for specially chosen functions $\varphi\:[0,\infty)\to[0,\infty)$, the mapping $f_\varphi\:V\to W_V$ given by the formula $f_\varphi(v)=\bigl(v,\varphi(\|v\|)\bigr)$ is an $\e$-isometry for some $\e\ge0$.

\begin{prop}\label{prop:isometry2}
For every positive $\e$ there exists a continuous $\e$-isometry $f_\e\:\ell_2\to\ell_2$ such that $d_H\bigl(f_\e(\ell_2),\,L\bigr)=\infty$ for every affine subspace $L\ss\ell_2$.
\end{prop}

\begin{proof}
It is clear that if we take the Euclidean norm $\|\cdot\|$ on the plane, then for the Hilbert space $V=\ell_2$ there is a standard linear isometry $H\:W_{\ell_2}\to\ell_2=V$. The desired mapping is now given by the formula $f_\e=H\circ f_\varphi\:V\to W_V\to V$, where the required function $\varphi$ is presented explicitly in the last paragraph of the paper~\cite{HU45}.
\end{proof}

\begin{prop}\label{prop:isometryinfty}
There exists an isometry $f\:\ell_\infty\to\ell_\infty$ such that $d_H\bigl(f(\ell_\infty),\,L\bigr)=\infty$ for every affine subspace $L\ss\ell_\infty$.
\end{prop}

\begin{proof}
It is clear that if we take the $\sup$-norm as the $\|\cdot\|$ norm on the plane, then the space $V=\ell_\infty$ has a standard linear isometry $H\:W_{\ell_\infty}\to\ell_\infty=V$. The desired mapping is now given by the formula $f=H\circ T_\varphi\:V\to W_V\to V$, where the required function $\varphi(t)=|t|$ is presented in the last paragraph of the paper~\cite{FSV02}.
\end{proof}

The authors believe that the central open question is~\cite[question (II) on p.4996, Problem 1]{CZ14},~\cite[Problem 3.2, 4.8(i)]{DD14}.

\begin{pr}\label{pr:isometry}
For two Banach spaces $V$, $W$, if there exists an $\e$-isometry $f\:V\to W$ for some $\e>0$, does $W$ contain an isometric\/ \(not necessarily linear\/\) copy of $V$\rom?
\end{pr}

D.~Dai and Y.~Dong~\cite[Corollary 3.4]{DD14}, assuming that $V$ is separable and $W$ is reflexive, showed the existence of a linear isometric embedding of $V$ into $W$. L.~Cheng, Q.~Cheng, K.~Tu, and J.~Zhang~\cite[Lemma 2.1]{CCTZ15} constructed an embedding into a larger space.

\begin{thm}\label{thm:CCTZ}
Suppose that $f\:V\to W$ is an $\e$-isometry and $\xi$ is a free ultrafilter on $\N$, i.e., $\xi\in\N^{*}=\beta\N\sm\N$ is a point in the growth of natural numbers in the compact Stone--\v{C}ech extension $\b\N$. Then
$$
\Phi(v)=w^{*}-\lim_{\xi}\frac{f(nv)}n\ \text{for all $v\in V$}
$$
defines an isometry $\Phi\:V\to W^{**}$.
\end{thm}
We emphasize that the isometry $\Phi\:V\to W^{**}$ need not be affine. For the mapping $f$ from Proposition~\ref{prop:isometryinfty}, the equality $\Phi=f$ holds. But for a uniformly convex Banach space $W$, for any $\e$-isometry $f\:V\to W$ with $f(0)=0$, the limit
$$
\varphi(v)=\lim_{s\to\infty}\frac{f(sv)}s
$$
exists for all $v\in V$, and the map $\varphi\:V\to W$ is a linear isometry~\cite[Proposition 2.3]{SV03}.

From the theorem of A.~Dvoretsky~\cite{D61},~\cite{M71} on almost spherical sections it follows that \emph{in any infinite-dimensional vector space $W$ for any natural $n$ there exists a sequence of subspaces $L_i^n\ss W$ such that
$$
d_{BM}\bigl(\ell_2^n,L_i^n\bigr)\-\mathop{\to}\limits_{i\to\infty}0\ \ \text{ and }\ \ d_K\bigl(\ell_2^n,\,L_i^n\bigr)\mathop{\to}\limits_{i\to\infty}0.
$$
}
But both distances $d_{BM}$ and $d_K$ have a multiplicative character and the formulated statement does not imply (linear) isometric embeddability of $\ell_2^n$ in $W$.

Let us now present some results on approximations with a small deviation of a multiplicative nature.

To clarify the importance of replacing coincidence of objects with their approximate coincidence in the age of computer data processing, we quote from~\cite[p. 270)]{SV03}. ``The notion of $\e$-isometries was motivated by the fact that observations in the real world always contain some small error. Thus, a mapping $f$ that assigns real-world points to points in a mathematical model is always an $\e$-isometry. But if we cannot measure distances precisely, then we cannot check whether $f$ is surjective. Therefore, we believe that it is more natural to study mappings that are both $\e$-isometric and $\e$-surjective. Clearly, Theorems~\ref{thm:D} and~\ref{thm:equivalent} allow some ``exact'' statements to be translated into ``approximate'' ones.

\begin{cor}\label{cor:compactspaces}
For compact Hausdorff spaces $X$ and $Y$, the following conditions are equivalent\/\rom:
\begin{enumerate}
\item[$(0)$]\namedlabel{cor:compactspases:0}{(0)} the spaces $X$ and $Y$ are homeomorphic\/\rom;
\item[$(3')$] $d_{GH}^c\bigl(C(X),C(Y)\bigr)<\infty$\rom;
\item[$(4')$] $d_{GH}\bigl(C(X),C(Y)\bigr)<\infty$\rom;
\item[$(5)$] the unit balls $B_{C(X)}$ and $B_{C(Y)}$ are isometric\rom;
\item[$(6')$]\namedlabel{cor:compactspases:6'}{(6')} $d_{BM}\bigl(C(X),C(Y)\bigr)<\log2$\rom;
\item[$(9')$]\namedlabel{cor:compactspases:9'}{(9')} $d_{GH}(B_V,B_W)<1/5$.
\end{enumerate}
\end{cor}

\begin{proof}
The equivalence $\ref{cor:compactspases:0}$\Llra$\ref{cor:compactspases:6'}$ is just the theorem that was proved independently by D.~Amir~\cite{A65} and M.~Cambern~\cite{C67}.

The equivalence $\ref{cor:compactspases:0}$\Llra$\ref{cor:compactspases:9'}$ was proved in 2011 by R.~G\'orak~\cite{G11}.
\end{proof}

In 2005, Y.~Dutrieux and N.J.~Kalton~\cite{DK05} proved Corollary~\ref{cor:compactspaces} with more strong requirements ``$(9')$ $d_{GH}(B_V,B_W)<1/16$'' and ``$(8')$ $d_{K}(V,W)<1/10$''. In the condition~$\ref{cor:compactspases:6'}$ the constant $\log 2$ cannot be decreased~\cite{C75}. In~\cite{DK05} the case of countable compact spaces is considered separately, for which weaker conditions $(6')$, $(8')$, and $(9')$ are found. In~\cite{G70} and~\cite{CG13} it is shown that for countable compact spaces, one can take $\log 3$ in~$\ref{cor:compactspases:6'}$. In~\cite{CG12} the case of a discrete space $X$ is considered.

For normed spaces of finite dimension $n$, conditions like $(6')$, $(8')$, and $(9')$ cannot be included in Theorem~\ref{thm:equivalentfinite}. \emph{The Banach--Mazur space $Q(n)$ of all classes of isometric $n$-dimensional normed spaces, considered with the Banach--Mazur distance $d_{BM}$, is compact}, see~\cite[p. 544]{WOP90}. In~\cite[1.2]{K75-1}, the space $Q(n)$ is called the \emph{Minkowski space}.

In 1948, F.~John~\cite{J48} proved that \emph{for every space $V\in Q(n)$, the inequality $d_{BM}(\ell_2^2,\,V)\le\log\sqrt n$ holds, and this upper bound is sharp\/\rom: the equality sign is attained for $\ell_1^n$}. In 1996, S.~Ageev, S.~Bogatyi, and P.~Fabel~\cite{F96, ABF96, ABF98} proved that the Banach--Mazur compactum $Q(n)$ is a compact absolute retract. The complement $Q_{\cE}(2)=Q(2)\sm\{\ell_2^2\}$ of the point corresponding to the Euclidean space $\ell_2^n$ is a $Q$-manifold and has nontrivial four-dimensional cohomology, so it is not contractible~\cite{AB98}. Therefore, the \emph{Banach--Mazur compactum $Q(2)$ is not homeomorphic to the Hilbert cube $Q=I^{\aleph_0}$}~\cite{AB98}. \emph{The subset $Q_{\cE}(n)=Q(n)\sm\{\ell_2^n\}\ss Q(n)$ is a $Q$-manifold for any finite $n\ge2$}~\cite{ABRF98},~\cite{ABR03}. Since the metric space $\bigl(Q(n),\,d_{BM}\bigr)$ is compact, it is not difficult to deduce from here and from Theorem~\ref{thm:equivalentfinite} that the metrics $d_{BM}$, $d_{GH}^c$, $d_K$, and $d_{GH}$ generate the same topology on $Q(n)$.

The condition~$\ref{thm:equivalentfinitedistance:5'}$ from Theorem~\ref{thm:equivalentfinitedistance} and Problem~\ref{pr:isometry} can be connected as follows.

\begin{pr}\label{pr:boll}
Is it true that if for Banach spaces $V$ and $W$ there is an isometric embedding $f\:B_V\to W$, then there is a\/ \(linear\/\) isometric embedding $F\:V\to W$ \(such that $F|_{B_V}=f$\)\rom?
\end{pr}

For a Banach space $V$, let $Q(V)$ denote the space of all Banach spaces isomorphic to $V$, considered with the Banach--Mazur metric $d_{BM}$.

\begin{pr}\label{pr:BM}
Is it true that for non-isomorphic Banach spaces $V$ and $W$, the spaces $Q(V)$ and $Q(W)$ are not isometric\/ \(not homeomorphic\/\)\rom?
\end{pr}

\begin{prop}
If the compact sets $Q(2)$ and $Q(n)$ are isometric, then $n=2$.
\end{prop}

\begin{proof}
Since all one-dimensional normed spaces are isometric, it follows that $\bigl|Q(1)\bigr|=1$. Therefore, we can assume that $n\ge2$. Let $f\:Q(2)\to Q(n)$ be an isometry. Since at the point $\{\ell_2^2\}\in Q(2)$ the space $Q(2)$ is locally not homeomorphic to the Hilbert cube, and the space $Q_{\cE}(n)=Q(n)\sm\{\ell_2^n\}$ is a $Q$-manifold, it follows that $f\bigl(\{\ell_2^2\}\bigr)=\{\ell_2^n\}$. Thus,
$$
f\bigl(Q(2)\bigr)=f\bigl(B(\{\ell_2^2\},\,\log\sqrt2)\bigr)=B(\{\ell_2^n\},\,\log\sqrt2)\ss Q(n).
$$
The equality $\log\sqrt2=\log\sqrt n$ implies the required equality $2=n$.
\end{proof}
The estimates of the diameter of compact sets $Q(n)$ known to the authors do not allow to distinguish the sets metrically.

Alternative proofs of the Amir-Cambern theorem were presented by H.B.~Cohen~\cite{C76} and L.~Drewnowski~\cite{D88}.

Note that J.~Benyamini~\cite{B81} generalized the Amir--Cambern theorem in the spirit of the multiplicative version of the Hyers--Ulam theorem. \emph{If $X$ is a compact metric space, $0<\e<1$, and $f$ is a linear mapping from $C(X)$ to $C(Y)$ satisfying the condition
\begin{equation}\label{eq:6}
\|\varphi\|\le\bigl\|f(\varphi)\bigr\|\le(1+\e)\|\varphi\|\ \ \text{for all $\varphi\in C(X)$},
\end{equation}
then, firstly, there exists a continuous mapping $f$ from a closed subset $Y_1$ of $Y$ onto $X$\rom; secondly, there exists a linear isometry $U$ of $C(X)$ into $C(Y)$ such that
$$
U(\varphi)(y)=\varphi\circ f(y)\ \text{ for all $y\in Y_1$},
$$
and
$$
\|Uf\|\le3\e.
$$
}
The example given by Y.~Beniamini proves that the second part of his theorem does not hold in general for non-metric spaces. K.~Jarosz showed that the first part of Beniamini's theorem also holds for non-metrizable spaces. The case $\e=0$ corresponds to Holszty\'nski's theorem~\cite{H66}.

As a work considering in~(\ref{eq:6}) not necessarily a linear mapping, we present~\cite{J89}. In~\cite{GS18} a combined additive-multiplicative version of the approximation theorem is proposed.

%%%%%%%%%%%%%%%%%%%%%%%%%%%%%%%%%%%%%%%%%%%%%%

\end{document}